\documentclass{amsart}
\allowdisplaybreaks[4]
\numberwithin{equation}{section}
\DeclareMathOperator{\td}{d}
\usepackage{amssymb,ifpdf}
\ifpdf
\usepackage[hyperindex,pagebackref]{hyperref}
\else
\expandafter\ifx\csname dvipdfm\endcsname\relax
\usepackage[hypertex,hyperindex,pagebackref]{hyperref}
\else
\usepackage[dvipdfm,hyperindex,pagebackref]{hyperref}
\fi
\fi
\theoremstyle{plain}
\newtheorem{theorem}{Theorem}[section]
\newtheorem{corollary}{Corollary}[theorem]
\theoremstyle{remark}
\newtheorem{remark}{Remark}[section]

\begin{document}

\title[Extended fermionic $p$-adic integrals on $\mathbb{Z}_p$]
{Extended fermionic $p$-adic integrals on $\mathbb{Z}_p$}

\author[F. Qi]{Feng Qi}
\address[Qi]{Department of Mathematics\\ College of Science\\ Tianjin Polytechnic University\\ Tianjin City, 300387\\ China}
\email{\href{mailto: F. Qi <qifeng618@gmail.com>}{qifeng618@gmail.com}, \href{mailto: F. Qi <qifeng618@hotmail.com>}{qifeng618@hotmail.com}, \href{mailto: F. Qi <qifeng618@qq.com>}{qifeng618@qq.com}}
\urladdr{\url{http://qifeng618.wordpress.com}}

\author[S. Araci]{Serkan Araci}
\address[Araci]{Department of Mathematics\\ Faculty of Science and Arts\\ University of Gaziantep\\ 27310 Gaziantep\\ Turkey}
\email{\href{mailto: S. Araci <mtsrkn@hotmail.com>}{mtsrkn@hotmail.com}}

\author[M. Acikgoz]{Mehmet Acikgoz}
\address[Acikgoz]{Department of Mathematics\\ Faculty of Science and Arts\\ University of Gaziantep\\ 27310 Gaziantep\\ Turkey}
\email{\href{mailto: M. Acikgoz <acikgoz@gantep.edu.tr>}{acikgoz@gantep.edu.tr}}

\begin{abstract}
In the paper, using the extended fermionic $p$-adic integral on $\mathbb{Z}_p$, the authors find some applications of the umbral calculus. From these applications, the authors derive some identities on the weighted Euler numbers and polynomials. In other words, the authors investigate systematically the class of Sheffer sequences in connection with the generating function of the weighted Euler polynomials.
\end{abstract}



\maketitle

\section{{Preliminaries}}

Let  $\mathbb{C}$ denote the set of complex numbers, $\mathcal{F}$ the set of all formal power series in $t$ over $\mathbb{C}$ with
\begin{equation}
\mathcal{F}=\Biggl\{ f(t)=\sum_{k=0}^\infty a_k\frac{t^k}{k!}\Bigg|a_k\in\mathbb{C}\Biggr\},
\end{equation}
$\mathcal{P}=\mathbb{C}[x] $, $\mathcal{P}^{\ast}$ the vector space of all linear functionals on $\mathcal{P}$, and $\langle L|p(x)\rangle $ the action of the linear functional $L$ on the polynomial $p(x)$.
\par
It is well-known that the vector space operation on $\mathcal{P}^{\ast}$ is defined by
\begin{gather}
\langle L+M|p(x)\rangle=\langle L|
p(x)\rangle+\langle M|p(x)
\rangle\quad\text{and}\quad
\langle cL|p(x)\rangle=c\langle L|
p(x)\rangle,
\end{gather}
where $c$ is a complex constant.
\par
The formal power series is defined by
\begin{equation}
f(t)=\sum_{k=0}^\infty a_k\frac{t^k}{k!}\in\mathcal{F}
\end{equation}
which describes a linear functional on $\mathcal{P}$ as $\langle f(t)|x^n\rangle=a_{n}$ for all $n\ge 0$. In particular,
\begin{equation}
\bigl\langle t^k|x^n\bigr\rangle=n!\delta_{n,k},
\label{euation-2}
\end{equation}
where $\delta_{n,k}$ stands for the Kronecker delta. If we take
\begin{equation}
f_{L}(t)=\sum_{k=0}^\infty\bigl\langle L|x^k\bigr\rangle\frac{t^k}{k!},
\end{equation}
then
\begin{equation}
\langle f_{L}(t)|x^n\rangle=\langle
L|x^n\rangle.
\end{equation}
\par
Additionally, the map $L\to f_{L}(t)$ is a vector space isomorphism
from $\mathcal{P}^{\ast}$ onto $\mathcal{F}$. Henceforth, $\mathcal{F}$
will denote both the algebra of the formal power series in $t$ and the
vector space of all linear functionals on $\mathcal{P}$. So an element $f(t)$ of $\mathcal{F}$ will be thought of as both a formal power series and a linear functional and $\mathcal{F}$ will be called an umbral algebra.
\par
It is well-known that $\langle e^{yt}|x^n\rangle=y^n$, which implies that
\begin{equation}
\langle e^{yt}|p(x)\rangle=p(y).
\end{equation}
\par
We note that for all $f(t)$ in $\mathcal{F}$
\begin{equation}
f(t)=\sum_{k=0}^\infty\bigl\langle f(t)|x^k\bigr\rangle\frac{t^k}{k!} \label{euation-3}
\end{equation}
and for all polynomials $p(x)$
\begin{equation}
p(x)=\sum_{k=0}^\infty\bigl\langle t^k|p(x)\bigr\rangle\frac{x^k}{k!}. \label{euation-4}
\end{equation}
\par
The order $o(f(t))$ of the power series $f(t)\ne 0$ is the smallest integer $k$ for which $a_k $ does not vanish. We say that $o(f(t))=\infty $ if $f(t)=0$. It is clear that
\begin{equation*}
o(f(t)g(t))=o(f(t))+o(g(t))\quad\text{and}\quad o(f(t)+g(t))\ge\min\{ o(f(t)),o(g(t))\}.
\end{equation*}
\par
A series $f(t)$ has a multiplicative inverse, denoted by $f(t)^{-1}$ or $\frac{1}{f(t)}$, if $o(f(t))=0$. Such a series is called an invertible series. A series $f(t)$ satisfying $o(f(t))=1$ is called a delta series. For $f(t),g(t)\in\mathcal{F}$, we have
$$
\langle f(t)g(t)|p(x)\rangle=\langle f(t)|g(t)p(x)\rangle.
$$
A delta series $f(t)$ has a compositional inverse $\Bar{f}(t)$ such that
$$
f\bigl(\Bar{f}(t)\bigr)=\Bar{f}(f(t))=t.
$$
\par
By~\eqref{euation-3}, it follows that
\begin{equation}
p^{(k)}(x)=\frac{\td^kp(x)}{\td x^k}=\sum_{\ell=k}^\infty\frac{\bigl\langle t^{\ell}| p(x)\bigr\rangle}{\ell!}\prod_{i=0}^{k-1}(\ell-i)x^{\ell-k}\label{euation-5}
\end{equation}
and
\begin{equation}
p^{(k)}(0)=\bigl\langle t^k|p(x)\bigr\rangle=\bigl\langle 1|p^{(k)}(x)\bigr\rangle. \label{euation-6}
\end{equation}
The relation~\eqref{euation-5} implies that
\begin{equation}
t^kp(x)=p^{(k)}(x)=\frac{\td^kp(x)}{\td x^k} \label{euation-7}
\end{equation}
and
\begin{equation}
e^{yt}p(x)=p(x+y). \label{euation-8}
\end{equation}
\par
Let $S_{n}(x)$ denote a polynomial with degree $n$. Let $f(t)$ be a delta series and $g(t)$ an invertible series. Then there exists a unique sequence $S_{n}(x)$ such that
$$
\bigl\langle g(t)f^k(t)|S_{n}(x)\bigr\rangle=n!\delta_{n,k}
$$
for all $n,k\ge 0$. Such a sequence $S_{n}(x)$ is called a Sheffer sequence for $(g(t),f(t))$ or say that $S_{n}(t)$ is Sheffer for $(g(t),f(
t))$.
\par
The Sheffer sequence for $(1,f(t))$ is called an associated sequence for $f(t)$ or say that $S_{n}(x)$ is associated to $f(t)$. The Sheffer sequence for $(g(t),t)$ is called an Appell sequence for $g(t)$ or say that $S_{n}(x)$ is Appell for $g(t)$.
\par
Let $p(x)\in\mathcal{P}$. Then
\begin{equation*}
\langle f(t)|xp(x)\rangle=\langle\partial_tf(t)|p(x)\rangle=\langle f'(t)|p(x)\rangle
\end{equation*}
and
\begin{equation*}
\langle e^{ty}+1|p(x)\rangle=p(y)+p(0).
\end{equation*}
Let $S_{n}(x)$ be Sheffer for $(g(t),f(t))$. Then
\begin{align}
h(t)&=\sum_{k=0}^\infty\frac{\langle h(t)|S_k(x)\rangle}{k!}g(t)f^k(t), \quad h(t)\in\mathcal{F},  \\
p(x)&=\sum_{k=0}^\infty\frac{\langle g(t)f^k(t)|p(x)\rangle}{k!}S_k(x), \quad p(x)\in\mathcal{P},  \\
\frac{e^{y\Bar{f}(t)}}{g\bigl(\Bar{f}(t)\bigr)}&=\sum_{k=0}^\infty S_k(y)\frac{t^k}{k!}, \quad y\in\mathbb{C}, \label{euation-10}\\
f(t)S_{n}(x)&=nS_{n-1}(x).
\end{align}
Moreover, we have
\begin{equation}
\langle f_1(t)f_{2}(t)\dotsm f_{m}(t)|x^n\rangle=\sum\binom{n}{i_1, \dotsc,i_{m}}
\prod_{j=1}^m\bigl\langle f_j(t)|x^{i_j}\bigr\rangle,
\label{euation-33}
\end{equation}
where $f_1(t),f_{2}(t), \dotsc,f_{m}(t)\in\mathcal{F}$ and the sum is taken over all nonnegative integers $i_1, \dotsc,i_{m}$ such that $i_1+\dotsm+i_{m}=n$.
\par
For details on the above knowledge, please refer to~\cite{Araci-4, Dere-1, Kim-2, Kim-11, Kim-12, Maldonado, Roman} and plenty of references therein.
\par
Let $p$ be a fixed odd prime number. In what follows, we use $\mathbb{Z}_p$ to denote the ring of $p$-adic rational integers, $\mathbb{Q}$ the field of rational numbers, $\mathbb{Q}_p$ the field of $p$-adic rational numbers, and $\mathbb{C}_p$ the completion of algebraic closure of $\mathbb{Q}_p$. Let $\mathbb{N}$ be the set of natural numbers and $\mathbb{N}^{\ast}=\{0\}\cup\mathbb{N}$. The $p$-adic absolute value is defined by $\vert p\vert_p=p^{-1}$. We also assume that $\vert q-1\vert_p<1$ is an indeterminate. Let $UD(\mathbb{Z}_p)$ be the space of uniformly differentiable functions on $\mathbb{Z}_p$. For $f\in UD(\mathbb{Z}_p)$, the fermionic $p$-adic integral on $\mathbb{Z}_p$ is defined by Kim (see~\cite{NFAA-12-144.tex, NFAA-12-144.tex-arXiv, Serkan-Fasc.Math.2013, Serkan-Fasc.Math.arXiv}) as
\begin{equation}
I_{-1}(f)=\int_{\mathbb{Z}_p}f(a)\td\mu_{-1}(a)=\lim_{n\to\infty}\sum_{a=0}^{p^n-1}f(a)(-1)^a.
\label{euation-34}
\end{equation}
Hence, we have
\begin{equation}
I_{-1}(f_1)+I_{-1}(f)=2f(0), \label{equation 4}
\end{equation}
where $f_1(a)=f(a+1)$. For detailed information on these notions, see~\cite{Araci-2, Araci-1, Acikgoz, Kim-10, Kim-6, Kim-9, Kim-5, Kim-3, Kim-4, Kim-7, Kim-8}.
\par
Now let us consider Kim's $p$-adic fermionic integral on $\mathbb{Z}_p$. For $\vert 1-w\vert_p<1$,
\begin{equation}
I_{-1}^{w}(f)=\int_{\mathbb{Z}_p}w^af(a)\td\mu_{-1}(a)=\lim_{n\to\infty}\sum_{a=0}^{p^n-1}w^af(a) (-1)^a, \label{equation 47}
\end{equation}
where $I_{-1}^{w}(f)$ is the extended fermionic $p$-adic integral on $\mathbb{Z}_p$.
Letting $f(x)=e^{t(x+a)}\in UD(\mathbb{Z}_p)$ in this equation yields
\begin{equation}
\int_{\mathbb{Z}_p}w^ae^{t(x+a)}\td\mu_{-1}(a)=\frac{2}{we^{t}+1}e^{tx}=\sum_{n=0}^\infty E_{n,w}(x)\frac{t^n}{n!},
\label{equation 35}
\end{equation}
where $E_{n,w}(x)$ is the weighted Euler polynomials defined in~\cite{Kim-13}. Specially, the quantity $E_{n,w}(0)=E_{n,w}$ is the weighted Euler numbers. The relation between weighted Euler numbers and weighted Euler polynomials is given by
\begin{equation}
E_{n,w}(x)=\sum_{\ell=0}^n\binom{n}{\ell}x^{\ell}E_{n-\ell,w}=(x+E_{w})^n, \label{equation 36}
\end{equation}%
with the usual convention of replacing $(E_{w})^n$ by $E_{n,w}$. Combing this with~\eqref{equation 35} leads to
\begin{equation}
E_{n,w}=\int_{\mathbb{Z}_p}w^aa^n\td\mu_{-1}(a)\quad\text{and}\quad E_{n,w}(x)=\int_{\mathbb{Z}_p}w^a(x+a)^n\td\mu_{-1}(a). \label{euation-11}
\end{equation}
\par
In~\cite{Dere-1, Dere-2}, the authors studied applications of the umbral algebra to special functions. In~\cite{Kim-1}, the author gave some new interesting links to works of many mathematicians in the analytic number theory and the modern classical umbral calculus. In~\cite{Kim-2, Kim-11, Kim-12}, the authors established some properties of the umbral calculus for Frobenius-Euler polynomials, Euler polynomials, and other special functions. In~\cite{Kim-1}, the authors investigated some new applications of the umbral calculus associated with $p$-adic invariants integral on $\mathbb{Z}_p$.
\par
In this paper, by the same motivation as in~\cite{Kim-1} and using the extended fermionic $p$-adic integral on $\mathbb{Z}_p$, we will give some applications of the umbral calculus and, from these applications, derive some identities concerning weighted Euler numbers, weighted Euler polynomials, and weighted Euler polynomials of order $k$.

\section{On the extended fermionic $p$-adic integral on $\mathbb{Z}_p$}

Now we start out to state and prove our main results.

\begin{theorem}\label{Proposition=1}
If $n\ge 0$, then $E_{n,w}(x)$ is an Appell
sequence for $g(t)=\frac{we^{t}+1}{2}$.
\end{theorem}

\begin{proof}
Suppose that $S_{n}(x)$ is an Appell sequence for $g(t)$. Then, by~\eqref{euation-10}, we have
\begin{equation}
\frac{1}{g(t)}x^n=S_{n}(x)\quad\text{if and only if}\quad x^n=g(t)S_{n}(x)
\label{euation-13}
\end{equation}
for $n\ge 0$.
Let
\begin{equation*}
g(t)=\frac{we^{t}+1}{2}\in\mathcal{F}.
\end{equation*}
It is clear that $g(t)$ is an invertible series. By~\eqref{euation-13}, we have
\begin{equation}
\sum_{n=0}^\infty E_{n,w}(x)\frac{t^n}{n!}=\frac{1}{g(t)}e^{xt}. \label{euation-14}
\end{equation}
This means that
\begin{equation}
\frac{1}{g(t)}x^n=E_{n,w}(x).
\label{euation-15}
\end{equation}
Making use of~\eqref{euation-10} gives
\begin{equation}
tE_{n,w}(x)=E'_{n,w}(x)=nE_{n-1,w}(x). \label{euation-16}
\end{equation}
Combining~\eqref{euation-15} and~\eqref{euation-16} results in Theorem~\ref{Proposition=1}.
\end{proof}

\begin{theorem}\label{Theorem=1}
Let $g(t)=\frac{we^{t}+1}{2}\in\mathcal{F}$.
Then for $n\ge 0$
\begin{equation}
E_{n+1,w}(x)=\biggl[x-\frac{g'(t)}{g(t)}\biggr]E_{n,w}(x). \label{euation-19}
\end{equation}
\end{theorem}

\begin{proof}
By~\eqref{euation-11}, we derive that
\begin{equation*}
\sum_{n=1}^\infty E_{n,w}(x)\frac{t^n}{n!}=\frac{xg(t)e^{xt}-g'(t)e^{xt}}{g(t)^{2}} =\sum_{n=0}^\infty\biggl[x\frac{1}{g(t)}x^n-\frac{g'(t)}{g(t)}\frac{1}{g(t)}x^n\biggr]\frac{t^n}{n!}.
\end{equation*}
Considering~\eqref{euation-15} and the above equality, we discover
\begin{equation*}
E_{n+1,w}(x)=xE_{n,w}(x)-\frac{g'(t)}{g(t)}E_{n,w}(x).
\end{equation*}
Theorem~\ref{Theorem=1} is thus proved.
\end{proof}

\begin{theorem}
For $n\ge 0$,
\begin{equation}
E_{n+1,w}(x)=\biggl[x-\frac{g'(t)}{g(t)}\biggr]E_{n,w}(x),
\label{euation-23}
\end{equation}
where $g'(t)=\frac{\td g(t)}{\td t}$.
\end{theorem}

\begin{proof}
From~\eqref{euation-11}, it is easy to see that
\begin{equation*}
\sum_{n=0}^\infty[wE_{n,w}(x+1)+E_{n,w}(x)]\frac{t^n}{n!}=\sum_{n=0}^\infty (2x^n)\frac{t^n}{n!}.
\end{equation*}
Comparing the coefficients on the both sides, we find
\begin{equation}
wE_{n,w}(x+1)+E_{n,w}(x)=2x^n.
\label{euation-21}
\end{equation}
From Theorem~\ref{Theorem=1}, it follows that
\begin{equation}
g(t)E_{n+1,w}(x)=g(t)xE_{n,w}(x)-g'(t)E_{n,w}(x) \label{euation-22}
\end{equation}
and
\begin{equation*}
(we^{t}+1)E_{n+1,w}(x)=(we^{t}+1)xE_{n,w}(x)-we^{t}E_{n,w}(x).
\end{equation*}
Consequently, we have
\begin{multline*}
wE_{n+1,w}(x+1)+E_{n+1,w}(x)\\
=w(x+1)E_{n,w}(x+1)+xE_{n,w}(x)-wE_{n,w}(x+1). \label{equation 37}
\end{multline*}
Combining this with~\eqref{euation-21} and~\eqref{euation-22}, we acquire the required conclusions.
\end{proof}

\begin{corollary}
\bigskip For $n\ge 0$, we have
\begin{equation*}
wE_{n+1}(x+1)+E_{n+1,w}(x)=2x^{n+1}.
\end{equation*}
\end{corollary}

\begin{theorem}
For $n\ge 0$, we have
\begin{align}
\langle f(t)|p(x)\rangle&=\int_{\mathbb{Z}_p}w^ap(a)\td\mu_{-1}(a), \label{euation-24}\\
\biggl\langle\frac{2}{we^{t}+1}\bigg|p(x)\biggr\rangle&=\int_{\mathbb{Z}_p}w^ap(a)\td\mu_{-1}(a), \label{equation=41}\\
E_{n,w}&=\biggl\langle\int_{\mathbb{Z}_p}w^ae^{at}\td\mu_{-1}(a)\bigg|x^n\biggr\rangle. \label{equation=42}
\end{align}
\end{theorem}

\begin{proof}
Let us consider the linear functional $f(t)$ satisfying
\begin{equation}
\langle f(t)|p(x)\rangle=\int_{\mathbb{Z}_p}w^ap(a)\td\mu_{-1}(a) \label{equation=38}
\end{equation}
for all polynomials $p(x)$. Then we readily see that
\begin{equation*}
f(t)=\sum_{n=0}^\infty\frac{\langle f(t)
|x^n\rangle}{n!}t^n=\sum_{n=1}^\infty\biggl[\int_{\mathbb{Z}_p}w^aa^n\td\mu_{-1}(a)\biggr] \frac{t^n}{n!}=\int_{\mathbb{Z}_p}w^ae^{at}\td\mu_{-1}(a).
\end{equation*}
Thus, we have
\begin{equation}
f(t)=\int_{\mathbb{Z}_p}w^ae^{at}\td\mu_{-1}(a)=\frac{2}{we^{t}+1}.
\label{equation 40}
\end{equation}
Therefore, by~\eqref{equation=38} and~\eqref{equation 40}, we arrive at the theorem.
\end{proof}

\begin{theorem}
For $p(x)\in\mathcal{P}$, we have
\begin{equation}\label{euation-27}
\int_{\mathbb{Z}_p}w^ap(x+a)\td\mu_{-1}(a)=\int_{\mathbb{Z}_p}w^ae^{at}\td\mu_{-1}(a)p(x)
=\frac{2}{we^{t}+1}p(x).
\end{equation}
Equivalently,
\begin{equation}
E_{n,w}(x)=\int_{\mathbb{Z}_p}w^ae^{at}\td\mu_{-1}(a)x^n=\frac{2}{we^{t}+1}x^n. \label{equation 43}
\end{equation}
\end{theorem}

\begin{proof}
From~\eqref{euation-11} and~\eqref{equation=42}, we see that
\begin{equation}
\begin{split}
\sum_{n=0}^\infty \biggl[\int_{\mathbb{Z}_p}w^a(x+a)^n\td\mu_{-1}(a)\biggr] \frac{t^n}{n!}
&=\int_{\mathbb{Z}_p}w^ae^{(x+a)t}\td\mu_{-1}(a)\\
&=\sum_{n=0}^\infty\biggl[\int_{\mathbb{Z}_p}w^ae^{at}\td\mu_{-1}(a)x^n\biggr]\frac{t^n}{n!}. \label{euation-25}
\end{split}
\end{equation}
By this equality and~\eqref{euation-11}, we see that for $n\in\mathbb{N}^{\ast}$
\begin{equation}
E_{n,w}(x)=\int_{\mathbb{Z}_p}(x+a)^n\td\mu_{-1}(a)=\int_{\mathbb{Z}_p}w^ae^{at}\td\mu_{-1}(a) x^n. \label{euation-26}
\end{equation}
As a result, we obtain the theorem.
\end{proof}

\begin{theorem}\label{indeterminate-thm}
For $p(x)\in\mathcal{P}$ and $k\in\mathbb{N}$, we have
\begin{multline}
\underset{\text{$k$-times}}{\underbrace{\int_{\mathbb{Z}_p} \dotsm\int_{\mathbb{Z}_p}}}w^{a_1+\dotsm+a_k}p(a_1+\dotsm+a_k+x)\prod_{j=1}^k\td\mu_{-1}(a_j)\\*
=\biggl(\frac{2}{we^{t}+1}\biggr)^kp(x). \label{equation 31}
\end{multline}
In particular,
\begin{align*}
E_{n,w}^{(k)}(x)&=\biggl(\frac{2}{we^{t}+1}\biggr)^kx^n\\
&=x^n\int_{\mathbb{Z}_p}\dotsm\int_{\mathbb{Z}_p}w^{a_1+\dotsm+a_k}e^{(a_1+\dotsm+a_k)t} \prod_{j=1}^k\td\mu_{-1}(a_j).
\end{align*}
Consequently,
\begin{equation*}
E_{n,w}^{(k)}(x)\sim \biggl(\biggl(\frac{we^{t}+1}{2}\biggr)^k,t\biggr).
\end{equation*}
\end{theorem}

\begin{proof}
For $\vert 1-w\vert_p<1$, we consider the weighted Euler polynomials of order $k$. \begin{multline}\label{equation=44}
\underset{\text{$k$-times}}{\underbrace{\int_{\mathbb{Z}_p} \dotsm\int_{\mathbb{Z}_p}}}w^{a_1+\dotsm+a_k}e^{(a_1+\dotsm+a_k+x) t}\prod_{j=1}^k\td\mu_{-1}(a_j)\\
=\biggl(\frac{2}{we^{t}+1}\biggr)^ke^{xt}
=\sum_{n=0}^\infty E_{n,w}^{(k)}(x)\frac{t^n}{n!}.
\end{multline}
where $E_{n,w}^{(k)}(0)=E_{n,w}^{(k)}$ are the weighted Euler numbers of order $k$.
Accordingly,
\begin{multline}
\underset{\text{$k$-times}}{\underbrace{\int_{\mathbb{Z}_p}\dotsm \int_{\mathbb{Z}_p}}}w^{a_1+\dotsm+a_k}(a_1+\dotsm+a_k)^n\prod_{j=1}^k\td\mu_{-1}(a_j)
\label{euation-28}\\
=\sum_{i_1+\dotsm+i_k=n}\binom{n}{i_1, \dotsc,i_{m}} \prod_{j=1}^k\int_{\mathbb{Z}_p}w^{a_j}a_j^{i_j}\td\mu_{-1}(a_j)\\
=\sum_{i_1+\dotsm+i_k=n}\binom{n}{i_1, \dotsc,i_{m}}\prod_{j=1}^kE_{i_j,w}
=E_{n,w}^{(k)}.
\end{multline}
Thanks to~\eqref{equation=44} and~\eqref{euation-28}, we have
\begin{equation}
E_{n,w}^{(k)}(x)=\sum_{\ell=0}^n\binom{n}{\ell}x^{\ell}E_{n-\ell,w}^{(k)}. \label{equation 45}
\end{equation}
From~\eqref{euation-28} and~\eqref{equation 45}, we notice that $E_{n,w}^{(k)}(x)$ is a monic polynomial of degree $n$ with coefficients in $\mathbb{Q}$. For $k\in\mathbb{N}$, let us assume that
\begin{multline}\label{euation-29}
g^{(k)}(t)=\Biggl[ {{\int_{\mathbb{Z}_p}\dotsm\int_{\mathbb{Z}_p}}} w^{a_1+\dotsm+a_k}e^{(a_1+\dotsm+a_k)t} \prod_{j=1}^k\td\mu_{-1}(a_j)\Biggr]^{-1}\\
=\biggl(\frac{we^{t}+1}{2}\biggr)^k.
\end{multline}
From this, we see that $g^{(k)}(t)$ is an invertible series. Due to~\eqref{equation=44} and~\eqref{euation-29}, we readily derive that
\begin{align*}
\frac{1}{g^{(k)}(t)}e^{xt}&=\underset{\text{$k$-times}} {\underbrace{\int_{\mathbb{Z}_p}\dotsm\int_{\mathbb{Z}_p}}}w^{a_1+\dotsm+a_k} e^{(a_1+\dotsm+a_k+x)t}\prod_{j=1}^k\td\mu_{-1}(a_j)\\
&=\sum_{n=0}^\infty E_{n,w}^{(k)}(x)\frac{t^n}{n!}.
\end{align*}
Taking account of this and
\begin{equation}
tE_{n,w}^{(k)}(x)=nE_{n-1,w}^{(k)}(x) \label{equation 46}
\end{equation}
yields that $E_{n,w}^{(k)}(x)$ is an Appell sequence for $g^{(k)}(t)$. Theorem~\ref{indeterminate-thm} is proved.
\end{proof}

\begin{theorem}\label{back-thm-okay}
For $p(x)\in\mathcal{P}$, we have
\begin{multline}
\biggl\langle\int_{\mathbb{Z}_p}\dotsm\int_{\mathbb{Z}_p} w^{a_1+\dotsm+a_k}e^{(a_1+\dotsm+a_k) t}\prod_{j=1}^k\td\mu_{-1}(a_j)\biggl|p(x)\biggr\rangle\\
=\int_{\mathbb{Z}_p}\dotsm\int_{\mathbb{Z}_p}w^{a_1+\dotsm+a_k}p(a_1+\dotsm+a_k) \prod_{j=1}^k\td\mu_{-1}(a_j).
\end{multline}
Furthermore,
\begin{equation*}
\biggl\langle \biggl(\frac{2}{we^{t}+1}\biggr)^k\bigg|p(x)\biggr\rangle =\int_{\mathbb{Z}_p}\dotsm\int_{\mathbb{Z}_p}w^{a_1+\dotsm+a_k}p(a_1+\dotsm+a_k) \prod_{j=1}^k\td\mu_{-1}(a_j),
\end{equation*}
equivalently,
\begin{equation*}
E_{n,w}^{(k)}=\biggl\langle\int_{\mathbb{Z}_p}\dotsm\int_{\mathbb{Z}_p}w^{a_1+\dotsm+a_k} e^{(a_1+\dotsm+a_k)t}\prod_{j=1}^k\td\mu_{-1}(a_j)\bigg|x^n\biggr\rangle.
\end{equation*}
\end{theorem}

\begin{proof}
Let us take the linear functional $f^{(k)}(t)$ fulfilling
\begin{equation}
\bigl\langle f^{(k)}(t)|p(x)\bigr\rangle=\int_{\mathbb{Z}_p}\dotsm \int_{\mathbb{Z}_p}w^{a_1+\dotsm+a_k}p(a_1+\dotsm+a_k)\prod_{j=1}^k\td\mu_{-1}(a_j)
\label{euation-32}
\end{equation}
for all polynomials $p(x)$. Then
\begin{align*}
f^{(k)}(t)&=\sum_{n=0}^\infty\frac{\bigl\langle f^{(k)}(t)|x^n\bigr\rangle}{n!}t^n\\
&=\sum_{n=0}^\infty\Biggl[\int_{\mathbb{Z}_p}\dotsm\int_{\mathbb{Z}_p}w^{a_1+\dotsm+a_k} (a_1+\dotsm+a_k)^n\prod_{j=1}^k\td\mu_{-1}(a_j)\Biggr]\frac{t^n}{n!}\\
&=\int_{\mathbb{Z}_p}\dotsm\int_{\mathbb{Z}_p}w^{a_1+\dotsm+a_k}e^{(a_1+\dotsm+a_k)t} \prod_{j=1}^k\td\mu_{-1}(a_j)\\
&=\biggl(\frac{2}{we^{t}+1}\biggr)^k.
\end{align*}
Therefore, we procure Theorem~\ref{back-thm-okay}.
\end{proof}

\begin{remark}
From~\eqref{euation-33}, we notice that
\begin{multline*}
\biggl\langle\int_{\mathbb{Z}_p}\dotsm\int_{\mathbb{Z}_p}w^{a_1+\dotsm+a_k}e^{(a_1+\dotsm+a_k) t}\prod_{j=1}^k\td\mu_{-1}(a_j)\bigg|x^n\biggr\rangle\\
=\sum_{i_1+\dotsm+i_k=n}\binom{n}{i_1, \dotsc,i_{k}} \prod_{\ell=1}^{k} \biggl\langle\int_{\mathbb{Z}_p}w^{a_\ell}e^{a_\ell t}\td\mu_{-1}(a_\ell)\bigg| x^{i_\ell}\biggr\rangle.
\end{multline*}
Therefore, we have
\begin{equation*}
E_{n,w}^{(k)}=\sum_{i_1+\dotsm+i_k=n}\binom{n}{i_1, \dotsc,i_{k}}E_{i_1,w}\dotsm E_{i_k,w}.
\end{equation*}
\end{remark}

\begin{remark}
Our applications to the weighted Euler polynomials, the weighted Euler numbers, and the weighted Euler polynomials of order $k$ seem to be interesting, because evaluating at $w=1$ leads to Euler polynomials and Euler polynomials of order $k$ defined respectively by
\begin{equation*}
\sum_{n=0}^\infty E_{n}(x)\frac{t^n}{n!}=\frac{2}{e^{t}+1}e^{xt}\quad \text{and}\quad
\sum_{n=0}^\infty E_{n}^{(k)}(x)\frac{t^n}{n!}=\biggl(\frac{2}{e^{t}+1}\biggr)^ke^{xt}.
\end{equation*}
It is also well known that
\begin{equation*}
E_{n}(x)=\int_{\mathbb{Z}_p}(x+a)^n\td\mu_{-1}(a)
\end{equation*}
and
\begin{equation*}
E_{n}^{(k)}(x) =\int_{\mathbb{Z}_p}\dotsm\int_{\mathbb{Z}_p} (a_1+\dotsm+a_k+x)^n\prod_{j=1}^k\td\mu_{-1}(a_j).
\end{equation*}
See~\cite{Araci-2, Araci-3, Acikgoz, Kim-6, Kim-3} and related references therein.
\end{remark}

\end{document}